\documentclass[preprint,12pt]{elsarticle}
 \usepackage{graphicx}
\usepackage{amssymb}
 \usepackage{amsthm}
 \theoremstyle{plain}
 \newtheorem{theorem}{Theorem}[section]

 \newtheorem{lemma}{Lemma}[section]
 \newtheorem{proposition}{Proposition}[section]
 
\theoremstyle{definition}
 \newtheorem{definition}{Definition}[section]
 
 \newtheorem{remark}{Remark}
\theoremstyle{remark}
\usepackage[english]{babel}
\usepackage{amssymb,amstext,amsmath}
\usepackage{amsfonts,amsmath,amssymb,amscd}
\usepackage{amsfonts}
\usepackage{mathrsfs}
 \usepackage{lineno}
\journal{arXiv}

\begin{document}

\begin{frontmatter}
 \title{A weak law of large numbers for the sequence of uncorrelated fuzzy random variables }
\author[label3]{Li Guan}
\address[label3]{College of Statistics and Data Science, Faculty of Science, Beijing University of Technology,100 Pingleyuan, Chaoyang District, Beijing, 100124,
P.R.China}
\ead{guanli@bjut.edu.cn}
\author[label1]{Jinping Zhang*\footnote{*Corresponding author: Jinping Zhang}}
 \address[label1]{Department of Mathematics and Physics, North China Electric Power University,
Beijing, 102206, P.R.China}
\ead{zhangjinping@ncepu.edu.cn}
\author[label2]{Jieming Zhou}
\address[label2]{MOE-LCSM, School of Mathematics and Statistics, Hunan Normal University, Changsha, Hunan, 410081, P. R. China}
\ead{jmzhou@hunnu.edu.cn}
\begin{abstract}
 We shall prove a weak law of large numbers for the uncorrelated (see Definition \ref{def:uncorrelated}) fuzzy random variable sequence with respect to the  uniform Hausdorff metric $d_H^{\infty}$, which  is an extension of weak law of large numbers for independent fuzzy random variables.
\end{abstract}


\end{frontmatter}


\section{Introduction}
\label{author_sec:1}
Limit theory is an important topic since sometimes we need to consider the asymptotic behavior or convergence property in applied fields such as stochastic control, mathematical finance, statistics, operational research and optimization  etc. The law of large numbers (LLN) is the important limit theorem with wide application in solving practical problems.

  Fuzzy random variable is the natural extension of random set (or set-valued random variable). A usual way to study  fuzzy random variable is to consider its $\alpha$-level sets, which is a set-valued random variable for each $\alpha$. The law of large numbers for set-valued and fuzzy random variables has received much attention since  Artstein and Vitale \cite{AV} proved the first LLN for compact random sets. For example, Hiai \cite{Hiai85} in 1985 proved the SLLN (strong law of large numbers) for set-valued random variables in the Mosco convergence. Uemura \cite{Uemura93} obtained a law of large numbers for random sets taking values in a class of subsets larger than the class of compact subsets of a Banach space.
  Detail review concerning  LLN for set-valued random variables earlier than 2002 can be referred to the book \cite{LiOgV}. Guan et al. \cite{Guan07} obtained SLLN for weighted sums of set-valued random variables in Rademacher type p Banach space. There are other references studied LLN for set-valued random variables. By using $\alpha$-level sets, some results for set-valued random variables were extended to the fuzzy case. For instance, Klement et al. \cite{Klement} (1986) obtained a SLLN for independent and identically distributed compact fuzzy random variables. Inoue \cite{Inoue91} (1991) studied SLLN for independent tight fuzzy random variables. Kim \cite{kim} proved a SLLN for independent and  identically distributed fuzzy random variables using a different metric. Guan and Li \cite{Li04} studied LLN  for weighted sums of fuzzy random variables. Li and Ogura \cite{LiOgsl} obtained the SLLN for independent (not necessarily identical distributed) fuzzy random variables. Ter$\acute{a}$n \cite{Teran2008} studied SLLN for t-normed arithmetics.

 In this paper, at first we propose the definition of
 uncorrelated fuzzy random variables (see Definition \ref{def:uncorrelated}) in the sense of level-wise by considering the $\alpha$-level set. For two fuzzy random variables, uncorrelation is weaker than independence. Under the weaker condition, then we shall prove a weak law of large numbers for the sequence of fuzzy random variables in the real line $\mathbb R$ with respect to the uniform Hausdorff metric $d_{H}^{\infty}$, which is different from the existing literature.

  The rest of the paper is organized as follows: Section \ref{sec2} contributes to preliminaries on set-valued and  fuzzy random variables. In Section \ref{sec3}, we shall present the main result.

\section{Preliminaries}
\label{sec2}

Throughout this paper, $(\Omega, {\mathscr A}, P)$ denotes a
nonatomic complete probability space. $\mathbb R$ is the set of real numbers. $K({\mathbb R})$ denotes the family of all nonempty closed subsets of $\mathbb R$. $K_k({\mathbb R})$ is the family of all nonempty compact subsets of ${\mathbb R}$, and $K_{kc}({\mathbb R})$ is the family of all nonempty compact convex subsets of ${\mathbb R}$.

For any  $A, B\in K({\mathbb R})$ and
$\lambda\in \mathbb R$, the addition and
scalar multiplication are defined as follows:
$$A+B=\{a+b:a\in A, b\in B\},$$
$$\lambda A=\{\lambda a:a\in A\}.$$

The Hausdorff metric on $K({\mathbb R})$ is defined by
$$
d_{H}(A,B)=\max\{\sup\limits_{a\in A}\inf\limits_{b\in B}|a-b|,\
\sup\limits_{b\in B}\inf\limits_{a\in A}|a-b|\}
$$
for $A,\ B\in K({\mathbb R})$. For $A\in K({{\mathbb R}})$,  define
$\|A\|_{\bf K}:=d_{H}(\{0\},A)$. It is known that the metric space $(K_k({\mathbb R}),d_H)$
is complete and separable, and $K_{kc}({\mathbb R})$ is a closed subset of
$(K_k({\mathbb R}),d_H)$ (cf. \cite{LiOgV}, Theorems 1.1.2 and 1.1.3).

Now we give a property of Hausdorff metric needed later, which appeared in \cite{LiOgsl} without given proof.
\begin{proposition}\label{proposition1}
Let $A_{1}\subset A_{2}\subset A_{3}$ and $B_{1}\subset B_{2}\subset B_{3}$. All of them belong to $K_{k}(\mathbb R)$. Then we have
$$d_{H}(A_2, B_2)\leq d_{H}(A_1, B_3)+d_{H}(A_3, B_1),$$
where $d_{H}(x, A)=\inf_{a\in A}|x-a|$ for $A\subset\mathbb R$.
\begin{proof}
 $d_{H}(A_{2}, B_{2})<\infty$ since both $A_2$ and $B_2$ are compact. By virtue of Theorem 1.1.14 in \cite{LiOgV},
$$
d_{H}(A_{2}, B_{2})=\sup_{x\in \mathbb R}\left\{|d
(x, A_2)-d(x, B_2)|\right\}.$$
For $x\in\mathbb R$, we have
$$
d(x, A_{2})-d_{H}(x, B_{2})\leq d(x, A_{1})-d(x, B_{3})
$$
and
$$
d(x, A_{2})-d_{H}(x, B_{2})\geq d(x, A_{3})-d(x, B_{1}).
$$
Then
$$
|d(x, A_{2})-d(x, B_{2})|\leq \max \left\{|d(x, A_{1})-d(x, B_{3})|, |d(x, A_{3})-d(x, B_{1})|\right\}.
$$
Furthermore,
\begin{equation}\label{eqn:0}
\begin{split}
d_{H}(A_{2}, B_{2})&=\sup_{x\in\mathbb R}\left\{|d(x, A_{2})-d(x, B_{2})|\right\}\\
&\leq \sup_{x\in\mathbb R}\max \left\{|d(x, A_{1})-d(x, B_{3})|, |d(x, A_{3})-d(x, B_{1})|\right\}\\
&\leq \sup_{x\in\mathbb R}\left\{|d(x, A_1)-d(x, B_3)|\right\}+\sup_{x\in\mathbb R}\left\{|d(x, A_3)-d(x, B_1)|\right\}\\
&=d_{H}(A_{1}, B_{3})+d_{H}(A_{3},B_{1}).
\end{split}
\end{equation}
\end{proof}
\end{proposition}
\remark  In general Banach space, the result also holds and further, from \eqref{eqn:0} we can get the stronger result
$$d_{H}(A_2, B_2)\leq \max \left\{d_{H}(A_1, B_3), d_{H}(A_3, B_1)\right\},$$
which was stated in \cite{Mol} without proof.

A set-valued mapping $F:\Omega\rightarrow K({\mathbb R})$ is called {\em
a set-valued random variable (or a random set)}, if for each open
subset $O$ of $X$, the inverse image $F^{-1}(O):=\{\omega \in \Omega: F(\omega)\cap O
\neq \emptyset \}$ belongs to ${\mathscr A}$.

The family of all integrable selections of $F$ is denoted by
$$S_F:=\left\{f\in L^1[\Omega;{\mathbb R}] :f(\omega)\in F(\omega) a.s.\right\},$$
where $L^1[\Omega;{\mathbb R}]$ is the family of all Lebesgue integrable (with respect to $P$) $\mathbb R$-valued functions.

A set-valued random variable $F$ is called {\em integrable} if $S_F$ is non-empty. It is called {\em integrably bounded}
 if
$\int_{\Omega}\|F(\omega)\|_{\bf K}dP <\infty$, which is equivalent to that $S_{F}$ is a bounded subset of $L^1[\Omega;{\mathbb R}]$ (cf. \cite{HiUm} or \cite{LiOgV}).
$L^1[\Omega, {\mathscr A},P; K({\mathbb R})]$ denotes the space of all integrably
bounded $K(\mathbb R)$-valued random variables. Similarly, we have notations $L^1[\Omega,{\mathscr A},P;  K_{k}({\mathbb R})]$ and $L^1[\Omega,{\mathscr A},P;  K_{kc}({\mathbb R})]$ respectively.

 For $F,G\in
L^1[\Omega,{\mathscr A},P;  K\left ({\mathbb R}\right)]$, $F=G$  means in the sense of $F(\omega)=G(\omega)~ a.s.$

For an integrable set-valued random variable $F$, its {\em expectation}, denoted by $E[F]$, is defined by Aumann in \cite{Au65} as following
$$
E[F]:=\Big\{\int_{\Omega}fdP:f\in S_{F}\Big\},
$$
where $\int_{\Omega}fdP$ is the usual Lebesgue integral.  $E[F]$ is also called the {\em Aumann integral} in literatures. Since here the underlying space is $\mathbb R$ and $(\Omega, \mathscr A, P)$ has no atom, it is known that the expectation $E[F]$ is a closed and convex subset of $\mathbb R$.

Let ${\mathcal F}_{k}({\mathbb R})$ be the family of all compact fuzzy sets:
$v:{\mathbb R}\rightarrow[0,1]$, where $v$ satisfies the following
conditions:\\
(1) The 1-level set $v_{1}=\{x\in {\mathbb R}:v(x)=1\}\neq \emptyset $.\\
(2) $v$ is upper semicontinuous, i.e. for each $\alpha\in
[0,1]$, the $\alpha$-level set $v_{\alpha}:=\{x\in
{\mathbb R}:v(x)\geq \alpha\}$ is a closed subset of ${\mathbb R}$.\\
(3) The support set $\textrm{cl}\{x\in {\mathbb R}:v(x)>0\}$ is
compact.

A fuzzy set $v$ in ${\mathcal F}_{k}({\mathbb R})$ is called {\em convex} if it satisfies
$$
v(\lambda x+(1-\lambda)y)\geq\min\{v(x),v(y)\},\ \mbox{for any}\ x,y
\in {\mathbb R}, \lambda\in [0,1].
$$
It is known that $v$ is convex if and only if each $\alpha$-level set
$v_{\alpha}~ (\alpha \in (0,1])$ is a convex subset of ${\mathbb R}$.
 ${\mathcal F}_{kc}({\mathbb R})$ denotes the class of all compact convex fuzzy sets.

The uniform metric $d_{H}^{\infty}$ (cf. \cite{PuRa86}) in ${\mathcal F}_{k}({\mathbb R})$ is defined as follows: for
$v^{1},v^{2}\in {\mathcal F}_{k}({\mathbb R})$,
$$
d_{H}^{\infty}(v^{1},v^{2}):=\sup\limits_{\alpha\in(0,1]}d_{H}(v^{1}_{\alpha},v^{2}_{\alpha}).
$$
Define the norm $\|v\|_{\bf
F}:=d_{H}^{\infty}(v,I_{0})=\sup_{\alpha>0}\|v_{\alpha}\|_{{\bf K}}$,
where $I_{0}$ is the indicator function of $\{0\}$. The space $({\mathcal F}_{k}({\mathbb R}),d_{H}^{\infty})$ is a complete
metric space  (cf. \cite{LiOg96}) but not separable in general (cf.
\cite{LiOgV}, Remark 5.1.7). Completeness was first proved by Puri
and Ralescu \cite{PuRa86} in the case of  the
d-dimensional Euclidean space $\mathbb R^d$.

It is well known that
$v_{\alpha}=\bigcap_{\beta<\alpha}v_{\beta}$, for every $\alpha\in
(0,1]$. We denote
$v_{\alpha+}=cl(\bigcup_{\beta>\alpha}v_{\beta}),$ for
$\alpha\in[0,1)$, which will be used later. Obviously, $v_{0+}$ is
 the support set of $v$. Due to the completeness of
$({\mathcal F}_{k}({\mathbb R}),d_{H}^{\infty})$, every Cauchy sequence
$\{v^{n}:n\in \mathbb{N}\}$ converges in ${\mathcal F}_{k}({\mathbb R})$ with respect to the metric $d_{H}^\infty$.

Now we present a result which will be used later.

\begin{lemma} \label{lem1} (cf. Lemma 2 of \cite{LiOgsl}) ~ Suppose a
sequence $\{v^{n}:n\in \mathbb{N}\}$ in ${\mathcal F}_{k}({\mathbb R})$ converges to $v$ in
${\mathcal F}_{k}({\mathbb R})$ with respect to $d_{H}^{\infty}$. Then for each
$\alpha\in[0,1)$, the sequence $\{v^{n}_{\alpha+}:n\in \mathbb{N}\}$
converges to a set $v_{\alpha^{*}}$ in $K_{k}({\mathbb R})$ with respect
to $d_{H}$. Further, $\lim_{\beta\downarrow\alpha}d_{H}(v_{\beta},
v_{\alpha^{*}})=0$, so that $v_{\alpha^{*}}=v_{\alpha+}$.
\end{lemma}

For any $v\in {\mathcal F}_{kc}({\mathbb R})$, define the support function of $v$ as
follows
$$s_v(x^*,\alpha)=\left\{\begin{array}{cc}
                        s(x^*,v_\alpha) & if\ \alpha>0,\\
                        s(x^*,v_{0+})   & if\ \alpha=0,\\
                        \end{array}\right.$$
for $(x^*,\alpha)\in S^*\times[0,1]$, where $S^*$ is the unit sphere of $\mathbb R^*$(
$\mathbb R^*=\mathbb R$ in the sense of isomorphism, but for the sake of clarity, we still use
$\mathbb R^*$ later) and $s(x^*, A)=\sup_{a\in A}x^*(a)$ for $x^*\in S^*$ and $A\subset \mathbb R$.

A mapping $X:\Omega\rightarrow {\mathcal F}({\mathbb R})$ is called a {\em fuzzy
set-valued random variable or a random upper semicontinuous
function} , if, for every $\alpha\in(0,1]$, $X_\alpha(\omega)=\{x\in
{\mathbb R}:X(\omega)(x)\geq\alpha\}$ is a set-valued random variable.

A  fuzzy random variable $X$ is called {\em integrably
bounded} if the real-valued random variable $\|X_{0+}(\omega)\|_{\bf
K}$ is integrable. Let $L^1[\Omega,{\mathscr A},P;{\mathcal F}_{k}({\mathbb R})]$ be the set
of all integrably bounded  fuzzy random variables and
$L^1[\Omega,{\mathscr A},P;{\mathcal F}_{kc}({\mathbb R})]$ be the set of all integrably
bounded  fuzzy random variables taking values in
${\mathcal F}_{kc}({\mathbb R})$. Two  fuzzy random variables $X, Y\in
L^1[\Omega,{\mathscr A},P;{\mathcal F}_{k}({\mathbb R})]$ are considered to be identical if
for any $\alpha\in[0,1], X_{\alpha}(\omega)=Y_{\alpha}(\omega)~
a.s.$

The {\em expectation} of a fuzzy random variable $X$,
denoted by $E[X]$, is an element in ${\mathcal F}_k({\mathbb R})$ such that, for every
$\alpha \in (0,1]$,
   $$
    (E[X])_{\alpha} = E[X_{\alpha}],
    $$
where the expectation of right hand side is the Aumann integral. From the
existence theorem (cf. \cite{LiOg96}), we can get an equivalent
definition: for any $x\in {\mathbb R}$,
        $$
         E(X)(x)=\sup\{\alpha\in [0,1]: x \in E[X_{\alpha}]\}.
        $$

  Note that $E[X]$ is always convex since $(\Omega, {\mathscr A}, P)$ is
nonatomic.

\section{Main Results}
\label{sec3}
\begin{definition}\label{def:uncorrelated}
Let $X^{1},X^{2}$ be  fuzzy random variables.  $X^{1}$ and $X^{2}$ are called {\em uncorrelated} if for any $\alpha\in (0,1]$, $X^1_\alpha$ and $X^2_\alpha$ are
 uncorrelated set-valued random variables. I.e. for each $x^*\in \mathbb R^*$, the real-valued random variables $s(x^*, X^1_\alpha)$ and $s(x^*, X^2_\alpha)$ are uncorrelated in the usual sense.

 Fuzzy random variables sequence  $X^1,X^2,\cdots$ are called {\em uncorrelated} if  the sequence $X^1_\alpha, X^2_\alpha,\cdots$ are pairwise uncorrelated for any $\alpha\in (0,1]$.
 \end{definition}

\begin{lemma}\label{lem2}
Let $\{X^n:n\in \mathbb N\}$ be a sequence of uncorrelated ${\mathcal F}_{kc}({\mathbb R})$-valued random variables. Then for each $\alpha\in(0,1]$, $\{X^n_{\alpha+}:n\in \mathbb N\}$ is a sequence of uncorrelated $K_{kc}({\mathbb R})$-valued random variables.
\end{lemma}

\begin{proof} Take a decreasing sequence $\{\alpha_j\}\subset (0,1]$ such that it
converges to $\alpha$. Then
$X^n_{\alpha+}=cl(\bigcup_jX^n_{\alpha_j}),\
X_{\alpha+}^m=cl(\bigcup_jX_{\alpha_j}^m)$ for $m, n\in \mathbb N$. By Lemma \ref{lem1},
 it holds that
$$\lim\limits_jd_H(X^n_{\alpha_j},X^n_{\alpha+})=0,\
\ \ \lim\limits_jd_H(X_{\alpha_j}^m,X_{\alpha+}^m)=0.$$
Furthermore, for each $x^*\in\mathbb R^*$
$$\lim\limits_js(x^*, X^n_{\alpha_j})=s(x^*, X^n_{\alpha+}),\
\ \ \lim\limits_js(x^*, X^m_{\alpha_j})=s(x^*, X^m_{\alpha+}).$$
By Definition \ref{def:uncorrelated}, we know that
$$Cov\Big(s(x^*,X^n_\alpha), s(x^*,X^m_\alpha)\Big)=0. $$
Therefore, by the monotone convergence theorem,  we have
\begin{eqnarray*}
Cov\Big(s(x^*,X^n_{\alpha+}), s(x^*,X^m_{\alpha+})\Big)&=&Cov\Big(\lim\limits_js(x^*, X^n_{\alpha_j}), \lim\limits_js(x^*, X^m_{\alpha_j})\Big)\\
&=&\lim\limits_jCov\Big(s(x^*, X^n_{\alpha_j}), s(x^*, X^m_{\alpha_j})\Big)\\
&=&0.
\end{eqnarray*}
That shows the uncorrelation of the sequence $\{X^n_{\alpha+}:n\in \mathbb N\}$ for any $\alpha\in(0,1]$.
\end{proof}

\begin{theorem} \label{thm:2}
Let $\{X^n:n\in \mathbb N\}$ be a sequence of uncorrelated ${\mathcal F}_{kc}(\mathbb R)$-valued random variables such that for each n,
  $Var(s(x^*,X^{n}_\alpha))$ exists and for any $x^*\in \mathbb R^*$ ,
\begin{equation}\label{eqn:condition}
\frac{1}{n^2}\sum_{k=1}^{n}Var(s(x^*,X^k_\alpha))\longrightarrow0\ as \  n\rightarrow\infty.
\end{equation}
Then
\begin{equation}\label{eqn:result1}
P\Big\{d_{H}^\infty\Big(\frac{1}{n}\sum\limits_{k=1}^n X^{k},\frac{1}{n}\sum\limits_{k=1}^n E[X^{k}]\Big)>\varepsilon\Big\}\longrightarrow0 \ as \  n\rightarrow\infty.
\end{equation}
\end{theorem}

\begin{proof} {\em step 1:}

Firstly we prove that for any $\alpha\in(0,1]$,
$$\frac{1}{n^2}\sum_{k=1}^{n}Var(s(x^*,X^k_{\alpha+}))\longrightarrow0\ as \  n\rightarrow\infty.$$
Take a decreasing sequence $\{\alpha_j\}\subset(0,1]$, which
converges to $\alpha$.
$X^n_{\alpha+}=cl(\bigcup_jX^n_{\alpha_j}).$ By
Lemma \ref{lem1}, we have
$$\lim\limits_jd_H(X^n_{\alpha_j},X^n_{\alpha+})=0.$$
and
$$\lim\limits_js(x^*,X^n_{\alpha_j})=s(x^*, X^n_{\alpha+}).$$
Therefore, by the monotone convergence theorem and the condition \eqref{eqn:condition}, we have
\begin{eqnarray*}
\frac{1}{n^2}\sum_{k=1}^{n}\lim\limits_jVar(s(x^*,X^n_{\alpha_j}))&=&\frac{1}{n^2}\sum_{k=1}^{n}Var(s(x^*, X^n_{\alpha+}))\\
&\longrightarrow&0\ as \  n\rightarrow\infty.
\end{eqnarray*}
By Lemma \ref{lem2}, we know that both $\{X^n_{\alpha}:n\in \mathbb N\}$ and $\{X^n_{\alpha+}:n\in \mathbb N\}$ are sequences of uncorrelated set-valued random variables.
Then by theorem 3.2 of \cite{Guan2020}, we have
\begin{equation}\label{eqn:1}
P\Big\{d_{H}\Big(\frac{1}{n}\sum\limits_{k=1}^n X^{k}_\alpha,\frac{1}{n}\sum\limits_{k=1}^n E[X^{k}_\alpha]\Big)>\varepsilon\Big\}\longrightarrow0 \ as \  n\rightarrow\infty.
\end{equation}
and
\begin{equation}\label{eqn:2}
P\Big\{d_{H}\Big(\frac{1}{n}\sum\limits_{k=1}^n X^{k}_{\alpha+},\frac{1}{n}\sum\limits_{k=1}^n E[X^{k}_{\alpha+}]\Big)>\varepsilon\Big\}\longrightarrow0 \ as \  n\rightarrow\infty.
\end{equation}

{\em step 2:}

Take $\varepsilon>0$, and apply Lemma 2 for
$v_{n}=\frac{1}{n}\sum\limits_{i=1}^{n}E[X^{i}]$.
Then we can find a sequence
$0=\alpha_{0}<\alpha_{1}<...<\alpha_{m}=1$ such
that
\begin{equation}\label{eqn:3}
d_{H}\Big((\frac{1}{n}\sum\limits_{i=1}^{n}E[X^{i}])_{\alpha_{k}},\
(\frac{1}{n}\sum\limits_{i=1}^{n}E[X^{i}])_{\alpha_{k-1}+}\Big)<\varepsilon.
\end{equation}
Hence, by virtue of monotone property of level sets and the
above results, we have
$$(\frac{1}{n}\sum\limits_{i=1}^{n}X^{i})_{\alpha}=\frac{1}{n}\sum\limits_{i=1}^{n}
X^{i}_{\alpha},\\
(\frac{1}{n}\sum\limits_{i=1}^{n}E[X^{i}])_{\alpha}=\frac{1}{n}
\sum\limits_{i=1}^{n}E[
X^{i}_{\alpha}],$$
 For
$\alpha_{k-1}<\alpha\leq \alpha_{k} $, we have
$$
(\frac{1}{n}\sum\limits_{i=1}^{n}X^{i})_{\alpha_{k-1}+}\supseteq(\frac{1}{n}\sum\limits_{i=1}^{n}X^{i})_{\alpha}
\supseteq(\frac{1}{n}\sum\limits_{i=1}^{n}X^{i})_{\alpha_{k}}
$$
and
$$
(\frac{1}{n}\sum\limits_{i=1}^{n}E[X^{i}])_{\alpha_{k-1}+}\supseteq(\frac{1}{n}\sum\limits_{i=1}^{n}E[X^{i})]_{\alpha}
\supseteq(\frac{1}{n}\sum\limits_{i=1}^{n}E[X^{i}])_{\alpha_{k}}.
$$
Then by Proposition \ref{proposition1}, we obtain
\begin{eqnarray*}
&&d_{H}\Big((\frac{1}{n}\sum\limits_{i=1}^{n}X^{i})_{\alpha},\
(\frac{1}{n}
\sum\limits_{i=1}^{n}E[X^{i}])_{\alpha}\Big)\\
 &&\leq
d_{H}\Big((\frac{1}{n}\sum\limits_{i=1}^{n}X^{i})_{\alpha_{k}},\
(\frac{1}{n}\sum\limits_{i=1}^{n}E[X^{i}])_{\alpha_{k-1}+}\Big)\\
&&\hspace{0.5cm}
+d_{H}\Big((\frac{1}{n}\sum\limits_{i=1}^{n}X^{i})_{\alpha_{k-1}+},\
(\frac{1}{n}\sum\limits_{i=1}^{n}E[X^{i}])_{\alpha_{k}}\Big)\\
&&\leq
d_{H}\Big(\frac{1}{n}\sum\limits_{i=1}^{n}X^{i}_{\alpha_{k}},\
\frac{1}{n}\sum\limits_{i=1}^{n}E[X^{i}]_{\alpha_{k}}\Big)\\
&&\hspace{0.5cm}+
d_{H}\Big(\frac{1}{n}\sum\limits_{i=1}^{n}X^{i}_{\alpha_{k-1}+},\
\frac{1}{n}\sum\limits_{i=1}^{n}E[X^{i}]_{\alpha_{k-1}+}\Big)\\
&&\hspace{1cm}+2
d_{H}\Big((\frac{1}{n}\sum\limits_{i=1}^{n}E[X^{i}])_{\alpha_{k}},\
(\frac{1}{n}\sum\limits_{i=1}^{n}E[X^{i}])_{\alpha_{k-1}+}\Big).\\
\end{eqnarray*}

Consequently,
\begin{eqnarray*}
&&d_{H}^{\infty}\left(\frac{1}{n}\sum\limits_{i=1}^{n}X^{i},
\frac{1}{n}\sum\limits_{i=1}^{n}E[X^{i}]\right)\\
&&=\sup\limits_{\alpha\in(0,1]}d_{H}\Big((\frac{1}{n}\sum\limits_{i=1}^{n}X^{i})_{\alpha},\
(\frac{1}{n}
\sum\limits_{i=1}^{n}E[X^{i}])_{\alpha}\Big) \\
&&\leq\max\limits_{1\leq k\leq
m}d_{H}\Big(\frac{1}{n}\sum\limits_{i=1}^{n}X^{i}_{\alpha_{k}},\
\frac{1}{n}\sum\limits_{i=1}^{n}E[X^{i}_{\alpha_{k}}]\Big)\\
&&\hspace{0.5cm}+\max\limits_{1\leq k\leq m}
d_{H}\Big(\frac{1}{n}\sum\limits_{i=1}^{n}X^{i}_{\alpha_{k-1}+},\
\frac{1}{n}\sum\limits_{i=1}^{n}E[X^{i}_{\alpha_{k-1}+}]\Big)\\
&&+2\max\limits_{1\leq k\leq m}d_{H}\Big((\frac{1}{n}\sum\limits_{i=1}^{n}E[X^{i}])_{\alpha_{k}},\
(\frac{1}{n}\sum\limits_{i=1}^{n}E[X^{i}])_{\alpha_{k-1}+}\Big)\\
\end{eqnarray*}
By  using \eqref{eqn:1}, \eqref{eqn:2} and \eqref{eqn:3}, for any given  positive number $\varepsilon$, we obtain
\begin{equation*}
P\Big\{d_{H}^\infty\Big(\frac{1}{n}\sum\limits_{k=1}^n X^{k},\frac{1}{n}\sum\limits_{k=1}^n E[X^{k}]\Big)>\varepsilon\Big\}\longrightarrow0 \ as \  n\rightarrow\infty.
\end{equation*}
\end{proof}
\section*{Concluding remark}
As a manner similar to the uncorrelated set-valued random variables, we proposed uncorrelated fuzzy random variables by considering the $\alpha$-level sets. Uncorrelation is weaker than independence. For the sequence of uncorrelated fuzzy random variables, we proved a weak law of large numbers, which is an extension of weak law of large numbers for independent fuzzy random variables. With the development of technology, complex and big data are produced and obtained. Fuzzy statistics is a nice tool to deal with complex data. We wish our result will be applicable in fuzzy statistics.
\section*{Acknowledgment}
 This work is partly supported by the National Social Science Fund of
China No.19BTJ017(Li Guan), Natural Science Foundation of Beijing Municipality (No.1192015) (Jinping Zhang).



%
%

\end{document}